\documentclass[10pt]{amsart}

\usepackage{pdfsync}

\usepackage{amsmath,amsthm,amssymb,amsfonts,stmaryrd}

\usepackage{subcaption}
\usepackage{graphicx}
\usepackage{epstopdf}
\input xy
\xyoption{all}

\usepackage{hyperref}

\theoremstyle{plain}
\newtheorem{thm}{Theorem}[section]
\newtheorem{prp}[thm]{Proposition}
\newtheorem{lem}[thm]{Lemma}
\newtheorem{cor}[thm]{Corollary}

\theoremstyle{definition}
\newtheorem{dfn}[thm]{Definition}

\theoremstyle{remark}
\newtheorem{rmk}[thm]{Remark}

\newtheorem{exm}[thm]{Example}

\newcommand{\Hom}{\mathrm{Hom}}

\newcommand{\Inn}{\mathrm{Inn}}

\newcommand{\id}{\mathrm{id}}

\newcommand{\Ker}{\mathrm{Ker}}

\newcommand{\res}{\mathrm{res}}

\newcommand{\Ab}{\mathsf{Ab}}

\newcommand{\SSet}{\mathsf{SSet}} 	

\newcommand{\GL}{\mathrm{GL}}

\newcommand{\ls}[2]{~{}^{#1}\!{#2}}

\newcommand{\CF}{\mathcal{F}}
\newcommand{\CG}{\mathcal{G}}

\newcommand{\CM}{\mathcal{M}}

\newcommand{\BF}{\mathbb{F}}

\newcommand{\BN}{\mathbb{N}}

\newcommand{\BZ}{\mathbb{Z}}

\begin{document}
\title[]{On the cohomology of pro-fusion systems}
\author[A. D\'iaz Ramos]{Antonio D\'iaz Ramos}
\email{adiazramos@uma.es}
\author[O. Garaialde Oca\~na]{Oihana Garaialde Oca\~na}
\email{oihana.garayalde@ehu.eus}
\author[N. Mazza]{Nadia Mazza}
\email{n.mazza@lancaster.ac.uk}
\author[S. Park]{Sejong Park}
\email{sejongpark@gmail.com}

\date{\today}

\begin{abstract}
We prove the Cartan-Eilenberg stable elements theorem and construct a
Lyndon-Hochschild-Serre type spectral sequence for pro-fusion
systems. As an application, we determine
the continuous mod-$p$ cohomology ring of $\GL_2(\BZ_p)$ for any odd
prime $p$. 
\end{abstract}

\maketitle

\section{Introduction}
\label{section:introduction}

Throughout, let $p$ denote a prime number. Fusion systems for finite groups and
compact Lie groups have been successfully defined as algebraic models
for their $p$-completed classifying spaces, see \cite{BLOsurvey}. For
profinite groups, fusion was first studied in \cite{GRS1999}. More
recently, fusions systems have been defined over pro-$p$
groups and are termed \emph{pro-fusion systems}
\cite{StancuSymonds2014}.  

To compute the mod-$p$ cohomology rings of finite groups and compact
Lie groups two tools stand out, namely, the well-known
Cartan-Eilenberg stable elements theorem, and the
Lyndon-Hochschild-Serre spectral sequence.
In the present work, we
study the corresponding tools for the continuous mod-$p$ cohomology ring
$H^*_c(\cdot;\BF_p)$ of pro-fusion systems, where the coefficients are
the trivial module $\BF_p$. If there is no confusion, we
write $H^*_c(\cdot)=H^*_c(\cdot;\BF_p)$ for brevity. Our first main
result deals with pro-saturated pro-fusion systems (see Definition \ref{def:prosaturated_profusion_system}).

\begin{thm}[Stable Elements Theorem for Pro-Fusion Systems]
\label{T:introduction_Cartan_Eilenberg}
Let $\CF$ be a pro-saturated pro-fusion system on a pro-$p$ group $S$, where
$\CF= \varprojlim_{i\in I} \CF_i$ and
$S=\varprojlim_{i\in I}S_i$. Then there is a ring isomorphism
\[
H^*_c(S)^\CF \cong \varinjlim_{i\in I} H^*(S_i)^{\CF_i}.
\]
\end{thm}

Here, $H^*_c(S)^\CF$ and $H^*(S_i)^{\CF_i}$ are the subrings of stable
elements for $S$ and its finite quotients $S_i$, respectively, see Definition \ref{D:GeneralStablElements}. If $\CF$ is finitely generated in the sense of Definition \ref{def:fin-gen}, the stable elements are determined via a finite number of conditions. As
usual, we define the cohomology ring of the pro-fusion system $\CF$
(resp. the finite fusion system $\CF_i$) to be the subring
$H_c^*(\CF):=H^*_c(S)^\CF$ (resp. $H^*(\CF_i):=H^*(S_i)^{\CF_i}$) of
$H^*_c(S)$ (resp. $H^*(S_i)$). 

For a profinite group $G$, we consider the profinite $p$-completion of
the classifying space of $G$, $BG_p$, following Morel \cite{Morel1996}, and we show - analogously to the finite case - that there is a ring isomorphism
\[
H^*_c(BG_p)\cong H^*_c(G).
\]
In addition, 
we write $\CF_S(G)$ for the pro-fusion system defined by the
conjugation action of $G$ on a Sylow $p$-subgroup $S$. 
It turns out that $\CF_S(G)$ is pro-saturated and saturated and that it is finitely generated if $S$ is open in
$G$. This is the case if $G$ is a compact $p$-adic analytic group;
for example if $G=\GL_n(\BZ_p)$. 

\begin{thm}[Stable Elements Theorem for Profinite Groups]\label{T:introduction_cohomology_profinitegroup}
Let $G$ be a profinite group. Then, there is a ring isomorphism
\[
H^*_c(G)\cong H^*_c(S)^{\CF_S(G)}.
\]
\end{thm}

For both finite and compact Lie groups, a spectral sequence can be
built under weaker hypotheses than for the Lyndon-Hochschild-Serre
spectral sequence, \cite{Diaz2014,Gonzalez2016}. We prove
a version of this result for the continuous mod-$p$ cohomology of pro-fusion
systems.  

\begin{thm}\label{T:introduction_SpectralSequence}
Let $\CF$ be a pro-saturated pro-fusion system on a pro-$p$ group $S$
and let $T\leq S$ be a strongly $\CF$-closed subgroup. Then there is a first quadrant
cohomological spectral sequence with second page
\[
E_2^{n,m}= H_c^n(S/T;H_c^m(T))^\CF,
\]
and which converges to $H^*_c(S)^\CF$.
\end{thm}

As first example application, we compute the cohomology ring
$H_c^*({\BZ_3}^{1+2}_+;\BF_3)$, where ${\BZ_3}^{1+2}_+$ is the
$3$-adic version of the finite extraspecial group $3^{1+2}_+$ (see Example \ref{ex:extraspecial}). Then we
determine the cohomology rings of the general linear groups of
dimension $2$ over the $p$-adic integers. 

\begin{thm} 
\label{T:introduction_cohoGL2p>=3}
We have:
\begin{enumerate}
\item[(a)] For $p=3$, $H_c^*(\GL_2(\BZ_3);\BF_3)\cong \BF_3[X]\otimes \Lambda(Z_1, Z_2,Z_3)$, with degrees $|Z_1|=1$, $|Z_2|=|Z_3|=3$ and $|X|=4$.
\item[(b)] For $p>3$, $H_c^*(\GL_2(\BZ_p);\BF_p)\cong \Lambda(Z_1,Z_2)$ with degrees $|Z_1|=1$, $|Z_2|=3$.
\end{enumerate}
\end{thm}

The result for $p=3$ was already obtained in \cite{Henn1998} employing
different tools. The classes of degrees $2p-3$ and $2p-2$ found by
Aguad\'e \cite[Corollary 1.2]{Aguade1980} in the ring
$H^*(\GL_2(\BZ/p);\BF_p)$ survive to $H_c^*(\GL_2(\BZ_p);\BF_p)$ only
for $p=3$. Note that both rings in the statement of Theorem
\ref{T:introduction_cohoGL2p>=3} are Cohen-Macaulay and hence, by
Benson-Carlson duality for $p$-adic analytic groups \cite[\S
  12.3]{Benson2004}, they satisfy the Poincar\'e duality after
quotienting out the polynomial parts. The dualizing degrees are $7$
and $4$, respectively, and they are obtained after a degree shift of $4$, which is
the $p$-adic dimension of $\GL_2(\BZ_p)$. 

\textbf{Remarks and notation:} In Theorems
\ref{T:introduction_Cartan_Eilenberg}, \ref{T:introduction_cohomology_profinitegroup} and
\ref{T:introduction_SpectralSequence}, taking stable elements with
respect to the category $\CF^\circ$ instead of $\CF$ gives isomorphic
rings. Here, $\CF^\circ$ is the full subcategory of $\CF$ with objects
the open subgroups of $S$. We write $H^*_c(\cdot)$ instead of
$H^*_c(\cdot;\BF_p)$ for the continuous mod-$p$ cohomology ring and
$H^*(\cdot)$ instead of $H^*(\cdot;\BF_p)$ for the singular and for the
discrete mod-$p$ cohomology rings.
For a group $G$ and an element $g\in G$, we
let $c_g$ denote the conjugation morphism which maps $x$ to
$c_g(x)=\ls gx=gxg^{-1}$. If there is no confusion, we denote the
elements in a quotient by the same symbols as the elements that they
represent. As usual in the context of profinite groups, subgroups are assumed to be closed, generation is considered topologically, and homomorphisms are continuous.

\textbf{Outline of the paper:} In Section~\ref{section:profusion_systems} we present the necessary background on pro-fusion systems, in Section~\ref{section:cohomologyandclassifyingspaces}, we
briefly discuss cohomology, classifying spaces and $p$-completion for profinite groups, in
Section~\ref{section:stable_elements_and_a_spectral_sequence}, we
prove Theorems~\ref{T:introduction_Cartan_Eilenberg},
~\ref{T:introduction_cohomology_profinitegroup}
and~\ref{T:introduction_SpectralSequence}, and in Section \ref{section:coho_GL2}, we
determine the cohomology ring of $\GL_2(\mathbb Z_p)$, for $p=3$ and $p>3$ separately.

\textbf{Acknowledgements:} We are grateful to Peter Symonds for
helpful conversations, in particular, those regarding stable elements,
see Remark \ref{rmk:stableelements_versions}. We are also grateful to
Jon Gonz\'alez-S\'anchez for useful advice concerning the computations 
in Section 5.  

\section{Pro-fusion systems} \label{section:profusion_systems} 

In this section, we briefly introduce pro-fusion systems on pro-$p$
groups, and we review the
known results that we will need in the sequel.
Loosely, pro-fusion systems on pro-$p$ groups generalize the fusion
systems on finite $p$-groups.
We refer the reader to~\cite{StancuSymonds2014} and to~\cite{BLOsurvey}
for more background and more details on the topic.

First, recall that a fusion system on a finite $p$-group $S$ is a
category whose objects are the subgroups of $S$ and the morphisms are
injective group homomorphisms subject to certain axioms. 

\begin{dfn}[Morphisms of fusion systems]
Let $\CF$ and $\CG$ be fusion systems on finite $p$-groups $S$ and $T$,
respectively. A {\em morphism of fusion systems} $(\alpha,A)\colon \CF
\to \CG$ consists of a group homomorphism $\alpha\colon S\to T$ and a
functor $A\colon \CF \to \CG$ satisfying the following properties.
\begin{enumerate}
\item $A(P) = \alpha(P)$ for each subgroup $P$ of $S$.
\item For each morphism $\varphi\colon P\to Q$ in $\CF$ we have a
  commutative diagram
\[
\xymatrix{
	P \ar[r]^-{\alpha} \ar[d]_{\varphi} & \alpha(P) \ar[d]^{A(\varphi)} \\
	Q \ar[r]^-{\alpha} & \alpha(Q).
}
\]
\end{enumerate}
\end{dfn}

Note that if $(\alpha,A)\colon \CF \to \CG$ is a morphism of fusion
systems, then the homomorphism $\alpha$ completely determines the
functor $A$. Explicitly, in the above definition, $A(\varphi)$ is
defined by the formula $A(\varphi)(\alpha(u)) = \alpha(\varphi(u))$
for $u\in P$. From this formula, we see immediately that if a
homomorphism $\alpha\colon S\to T$ defines a morphism of fusion
systems from $\CF$ to $\CG$, then $\Ker(\alpha)$ is a strongly
$\CF$-closed subgroup of $S$. Conversely, if $\alpha\colon S\to T$ is a
homomorphism such that $\Ker(\alpha)$ is strongly $\CF$-closed in $S$,
then for any $\CF$-morphism $\varphi\colon P\to Q$, the induced
homomorphism $\alpha_\ast(\varphi)\colon \alpha(P)\to \alpha(Q)$
sending $\alpha(u)$ to $\alpha(\varphi(u))$ for $u\in P$ is
well-defined. In this case, $\alpha$ defines a morphism of fusion
systems from $\CF$ to $\CG$ if and only if $\alpha_\ast(\varphi) \in
\CG$ whenever $\varphi\in \CF$.


Now we give a slightly expanded reformulation of pro-fusion systems on
pro-$p$ groups given in~\cite{StancuSymonds2014}.
Recall that a poset $(I,\geq)$ is {\em directed} if for all
$i,j\in I$ there exists $k\in I$ such that $k\geq i$ and $k\geq j$.

\begin{dfn}[Pro-fusion systems] \label{D:pro-fusion system}
Suppose that we have an inverse system of fusion systems $\CF_i$ on
finite $p$-groups $S_i$, indexed by a directed poset $I$. This means
that there is a functor from $I$ (as a category) to the category of
fusion systems on finite $p$-groups. Explicitly, we have a fusion
system $\CF_i$ on a finite $p$-group $S_i$ for all $i\in I$ and we have
morphisms of fusion systems $(f_{ij},F_{ij})\colon \CF_j \to \CF_i$
for all $i, j\in I$ with $j\geq i$ such that $f_{ii} = \id_{S_i}$ (and
hence $F_{ii} = \id_{\CF_i}$) for all $i\in I$, and such that
$f_{ij}f_{jk} =f_{ik}$ (and hence $F_{ij}F_{jk} = F_{ik}$) for all $i,j,k\in I$ with
$k\geq j\geq i$. Set $S = \varprojlim_{i\in I} S_i$, and let $f_i\colon
S\to S_i$ be the canonical homomorphisms and $N_i = \Ker(f_i)$ for all
$i\in I$.
Then $S$ is a pro-$p$ group and $\{ N_i \mid i\in I \}$ is a basis of open
neighborhoods of $1$ in $S$ such that $N_j \leq N_i$ if $j\geq i$. We
define a category $\CF$ as follows. The objects of $\CF$ are the
closed subgroups of $S$. If $P$ is a closed subgroup of $S$, we have
$P = \varprojlim_{i\in I} f_i(P)$. If $Q$ is another closed subgroup
of $S$, then the functors $F_{ij}\colon \CF_j \to \CF_i$ define an
inverse system of finite sets $\Hom_{\CF_i}(f_i(P), f_i(Q))$ and we
have a map 
\[
\begin{array}{ccl}
	\varprojlim_{i\in I} \Hom_{\CF_i}(f_i(P), f_i(Q)) &\xrightarrow{\theta_{P,Q}} &\Hom_c(P, Q) \\
	(\varphi_i)_{i\in I} &\mapsto & (\varphi\colon (x_i) \mapsto (\varphi_i(x_i)))_{i\in I}.
\end{array}
\]
The relation between the maps $\varphi_i$ and $\varphi$ can be summarized by the following commutative diagram
\[
\xymatrix@C=1in{
	P \ar[r]^{f_j}\ar[d]_{\varphi} \ar@/^1.5pc/[rr]^{f_i} & f_j(P) \ar[r]^{f_{ij}}\ar[d]^{\varphi_j} & f_i(P)\ar[d]^{\varphi_i} \\
	Q \ar[r]_{f_j} \ar@/_1.5pc/[rr]_{f_i} & f_j(Q) \ar[r]_{f_{ij}} & f_i(Q)
}
\]
where $i, j\in I$ with $j \geq i$: by the universal property of the
inverse limit, $\varphi$ is the unique homomorphism making the left
square of the above diagram commutative for all $j\in I$. Here
$\Hom_c(P,Q)$ denotes the set of continuous homomorphisms from $P$ to
$Q$ with respect to the profinite topology. Indeed $\varphi\colon P\to
Q$ is continuous because we have $\varphi(P\cap N_i) \subseteq Q\cap
N_i$ for all $i\in I$ by the above diagram. We set $\Hom_\CF(P,Q)$ to
be the image of the above map $\theta_{P,Q}$. Since $\theta_{P,Q}$ is
injective,  we may identify 
\[
	\Hom_\CF(P, Q) = \varprojlim_{i\in I} \Hom_{\CF_i}(f_i(P), f_i(Q)).
\]
It is straightforward to see that $\CF$ is indeed a category under the
usual composition of maps. We say that $\CF$ is a {\em pro-fusion system} on the pro-$p$ group $S$.
\end{dfn}

Pro-fusion systems satisfy all the axioms of fusion systems
(cf.\ \cite[Lemmas 2.9, 2.12]{StancuSymonds2014}). In particular,
every fusion system $\CF$ on a finite $p$-group is a pro-fusion system
because it is isomorphic to the pro-fusion system defined by the
constant inverse system $\{ \CF \}$. Also note that the inclusions
$\varphi(P\cap N_i) \subseteq Q\cap N_i$ show the continuity of the
$\CF$-morphisms $\varphi$, and imply that the $N_i$ are strongly
$\CF$-closed subgroups of $S$. 

The notion of morphisms of pro-fusion systems on pro-$p$ groups is
identical to that of fusion systems on finite $p$-groups, with the
additional requirements that homomorphisms be continuous and that
subgroups be closed. 

\begin{dfn}[Morphisms of pro-fusion systems]
Let $\CF$, $\CG$ be pro-fusion systems on pro-$p$ groups $S$, $T$,
respectively. A {\em morphism of pro-fusion systems} $(\alpha,A)\colon
\CF \to \CG$ consists of a continuous homomorphism $\alpha\colon S\to
T$ and a functor $A\colon \CF \to \CG$ satisfying the following
properties. 
\begin{enumerate}
\item $A(P) = \alpha(P)$ for each closed subgroup $P$ of $S$.
\item For each morphism $\varphi\colon P\to Q$ in $\CF$ we have a commutative diagram
\[
\xymatrix{
	P \ar[r]^-{\alpha} \ar[d]_{\varphi} & \alpha(P) \ar[d]^{A(\varphi)} \\
	Q \ar[r]^-{\alpha} & \alpha(Q)
}
\]
\end{enumerate}
\end{dfn}

Morphisms of pro-fusion systems can be composed in the obvious way,
thus forming the category of pro-fusion systems. As for fusion
systems, if $(\alpha,A)\colon \CF \to \CG$ is a morphism of pro-fusion
systems, then the continuous homomorphism $\alpha$ completely
determines the functor $A$, and a continuous homomorphism
$\alpha\colon S\to T$ defines a morphism of pro-fusion systems from
$\CF$ to $\CG$ if and only if $\Ker(\alpha)$ is strongly $\CF$-closed
in $S$ and for each morphism $\varphi\colon P\to Q$ in $\CF$ the
induced homomorphism $\alpha_\ast(\varphi)\colon \alpha(P)\to
\alpha(Q)$ belongs to $\CG$.  

\begin{exm}[Canonical morphisms]
\label{example:canonicalmorphisms}
Let $\CF$ be the pro-fusion system defined by an inverse system of
fusion systems $\CF_i$ on finite $p$-groups $S_i$. We use the
notations in Definition~\ref{D:pro-fusion system}. Since each $N_i =
\Ker(f_i)$ is an open strongly $\CF$-closed subgroup of $S$, the
canonical homomorphism $f_i\colon S\to S_i$ is continuous and sends
each morphism $\varphi = (\varphi_i)$ in $\CF$ to the morphism
$\varphi_i$ in $\CF_i$. Thus $f_i$ defines a morphism of pro-fusion
systems $(f_i,F_i)\colon \CF \to \CF_i$. We call the morphisms
$(f_i,F_i)\colon \CF \to \CF_i$ the {\em canonical morphisms}
associated to the pro-fusion system $\CF$. 
\end{exm}

The proof of the next proposition is easy and we omit it.

\begin{prp}[Pro-fusion systems as inverse limits]
Let $\CF$ be the pro-fusion system defined by an inverse system of
fusion systems $\CF_i$ on finite $p$-groups. Then $\CF = \varprojlim_i
\CF_i$ is the inverse limit of the $\CF_i$ in the category of pro-fusion
systems. 
\end{prp}

Saturation for pro-fusion systems may be defined in a similar fashion
to that in the finite case, see \cite[Definition
  2.16]{StancuSymonds2014}. The next discussion illustrates some of
the subtleties underneath the concepts of (pro-)saturation.

\begin{dfn}\label{def:prosaturated_profusion_system}
A {\em pro-saturated fusion system} is a pro-fusion system $\CF =
\varprojlim_{i\in I} \CF_i$ on a pro-$p$ group $S=\varprojlim_{i\in
  I}S_i$, where each $\CF_i$ is a saturated fusion system on the
finite $p$-group $S_i$. 
\end{dfn}

A profinite group or pro-fusion system is termed \emph{countably based} if it can be expressed as an inverse limit indexed by the set of natural numbers $\BN$. 

\begin{lem}[{\cite[Theorems 3.9]{StancuSymonds2014}}]
A pro-fusion system $\CF$ on a pro-$p$ group $S$ is countably based if
and only if $S$ is countably based (as pro-$p$ group).
\end{lem}

\begin{thm}[{\cite[Theorems 5.1, 5.2]{StancuSymonds2014}}]
If $\CF$ is a pro-saturated fusion system and $\CF$ is countably based, then $\CF$ is saturated.
\end{thm}

\begin{exm}
If $G$ is a profinite group with Sylow pro-$p$ subgroup $S$, then we define the pro-fusion system that $G$ induces on $S$ as the category $\CF_S(G)$ with objects the closed subgroups of $S$ and morphisms  all homomorphisms induced by conjugation by elements of $G$ \cite[Example 2.8]{StancuSymonds2014}. Then $\CF_S(G)$ is a pro-fusion system in the sense of Definition \ref{D:pro-fusion system}, it is pro-saturated, and it is saturated by \cite[Example 2.18]{StancuSymonds2014}.
\end{exm}


\section{Cohomology, classifying spaces and $p$-completion}
\label{section:cohomologyandclassifyingspaces}

In this section, we first discuss continuous cohomology of profinite
groups and some of its properties. Since we are concerned about $p$-local
phenomena, we always consider coefficients in the trivial module
$\BF_p$. We refer the reader to \cite{Wilson1998},
\cite{RibesZalesskii2010} for more details and background.
Then we discuss classifying spaces and $p$-completion for simplicial sets \cite{BousfieldKan1972}
and for profinite simplicial sets \cite{Morel1996},
\cite{Goerss1998}. 

Let $G$ be a profinite group and let $H^*_c(G)$ be its continuous
mod-$p$ cohomology ring. The group
$H^*_c(G)$ can be recovered from the finite images of $G$, i.e., if $G\cong
\varprojlim_{i\in I} G_i$ for some inverse system of finite groups
$G_i$, then for all $*\geq 0$,  
\begin{equation}\label{equ:cohomology_is_inductive_limit}
H_c^*(G)\cong  \varinjlim_{i\in I} H^*(G_i).
\end{equation}

If $G$ is profinite and $K$ is a closed normal subgroup of $G$, there
exists a first quadrant cohomological spectral sequence converging to
$H_c^*(G)$ \cite[\S 7.2]{RibesZalesskii2010}. It is the
Lyndon-Hochschild-Serre spectral sequence, LHS s.s. for short, 
\begin{equation}\label{equ:LHSss_profinite}
E_2^{n,m}=H_c^n(G/K;H_c^m(K))\Rightarrow H^{n+m}_c(G).
\end{equation}
It is easy to check that this is a spectral sequence of
$\BF_p$-algebras.
It is the profinite counterpart to the standard
Lyndon-Hochschild-Serre spectral sequence for abstract (discrete)
groups. Below, in Theorem
\ref{T:coho_extraspecial_padics}, we provide an example of computation using
\eqref{equ:cohomology_is_inductive_limit} and
\eqref{equ:LHSss_profinite}.

\begin{exm}\label{ex:extraspecial}
Consider the
extraspecial group of order $p^3$ and exponent $p$, 
\begin{equation}\label{equ:presentation_of_p1+2+}
p^{1+2}_+=\langle A,B,C\;|\;A^p=B^p=C^p=[A,C]=[B,C]=1,[A,B]=C\rangle.
\end{equation}
The mod-$p$ cohomology ring $H^*(p^{1+2}_+)$ was computed by Leary in
\cite{Leary1991}. Define ${\BZ_3}^{1+2}_+$ to be the $3$-adic version of
the extraspecial group $3^{1+2}_+$, namely
${\BZ_3}^{1+2}_+=(\BZ_3\times\BZ_3)\rtimes \BZ/3$, with $\BZ/3$-action
on $\BZ_3\times\BZ_3$ given by the integral matrix 
$\left(\begin{smallmatrix} 1& -3\\1 &-2\end{smallmatrix}\right)$.
\end{exm}

\begin{thm}\label{T:coho_extraspecial_padics}
The continuous mod-$3$ cohomology of the group ${\BZ_3}^{1+2}_+$ is given by
\[
H^*_c({\BZ_3}^{1+2}_+;\BF_3)\cong \BF_3[x']\otimes\Lambda(y,y',Y,Y')/\{yy',yY,y'Y',YY',yY'-y'Y\},
\]
with degrees $|y|=|y'|=1$, $|Y|=|Y'|=|x'|=2$.
\end{thm}
\begin{proof}
For $i\geq 1$, let $N_i=3^i\BZ\times3^i\BZ$ and let
$G_i=G/N_i=(\BZ/3^i \times \BZ/3^i)\rtimes \BZ/3$ be the
semidirect product with $\BZ/3$ acting via the integral matrix
$\left(\begin{smallmatrix} 1& -3\\1 &-2\end{smallmatrix}\right)$. In
  particular, $G_1\cong 3^{1+2}_+$ as described
  in~\eqref{equ:presentation_of_p1+2+} and 
  we have surjective group homomorphisms $G_{i+1}\twoheadrightarrow G_i$. We set $G$ to
  be the pro-$3$ group $\varprojlim_{i\in I} G_i\cong
  {\BZ_3}^{1+2}_+$. 

The graded $\BF_3$-modules $H^*(G_i)$ are known to be isomorphic for
all $i\geq 1$ \cite[Proposition
  5.8]{DiazGaraialdeGonzalezSanches2017}, but the rings
$\{H^*(G_i)\}_{i>1}$ are still unknown. Nevertheless, here we
determine the ring $H_c^*(G)$. 

Consider the LHS spectral sequence \eqref{equ:LHSss_profinite}
associated to the normal subgroup $K=\BZ_3\times \BZ_3$ of $G$. Note
that $H_c^*(K)$ is an exterior algebra on two generators of degree
$1$. Moreover, $G/K\cong \BZ/3$ acts on $H_c^1(K)$ via
$\left(\begin{smallmatrix} 1& 1\\0 &1\end{smallmatrix}\right)$ and
  trivially on $H_c^2(K)$. 
Using the results in \cite[Corollary 4(ii)]{Siegel1996}, we obtain the $\BF_3$-generators for the corner of $E^{*,*}_2$,

\[
\xymatrix@=0pt{
&&&&&\\
2& \overline{Y} & \overline{Y}\overline{y'}& \overline{Y}\overline{x'}  & \overline{Y}\overline{y'}\overline{x'} & \overline{Y}\overline{x'}^2 \\
1& \overline{y} & \overline{Y'} & \overline{y}\overline{x'}  & \overline{Y'}\overline{x'}  & \overline{y}\overline{x'}^2 \\
0&1& \overline{y'} &  \overline{x'}  & \overline{y'}\overline{x'}  & \overline{x'}^2 \\
&0 & 1 & 2&3&4 \\
\ar@{-}"1,3"+<-25pt,-6pt>;"5,3"+<-25pt,-6pt>;
\ar@{-}"4,2"+<-15pt,-7pt>;"4,6"+<10pt,-7pt>;
}
\]

Then, $\overline{y},\overline{y'},\overline{x'},\overline{Y},\overline{Y'}$
generate the bigraded algebra $E_2^{*,*}$. The homomorphism $\pi\colon G\to G_1$ 
induces a morphism of spectral sequences from the LHS s.s.~ of $G_1$
to that of $G$ and all five generators of $E^{*,*}_2$ are in the image
of this morphism. By~\cite[Theorem 5(i)]{Siegel1996}, the LHS spectral
sequence of $G_1$ collapses at the second
page, and we have $E_2=E_\infty$ for
the LHS spectral sequence of $G$. Consider now the induced morphism $\pi^*\colon
H^*(G_1)\to H_c^*(G)$ and define $y,y',x',Y,Y'$ to be the image by
$\pi^*$ of the generators of the same name in \cite[Theorem
  7]{Leary1991}. Then they give rise to the aforementioned overlined
generators of $E^{*,*}_2$, and they satisfy the relations: 
\begin{equation}\label{equ:relations_Hc(Z3Z3.C3)}
yy'=yY=y'Y'=YY'=Y^2=Y'^2=0\text{, }yY'=y'Y.
\end{equation}
The corresponding relations between the overlined generators give a
bigraded algebra isomorphic to $E^{*,*}_2$. Hence, $H^*_c(G)$ is
generated by the elements $y,y',x',Y,Y'$ of degrees $1,1,2,2,2$
respectively subject to the relations
\eqref{equ:relations_Hc(Z3Z3.C3)}. 
\end{proof}

The notion of continuous and discrete cohomology groups may be
extended to the categories $\widehat \SSet$ of simplicial profinite
sets and $\SSet$ of simplicial (discrete) sets respectively. In fact,
if $X=\{X_n\}_{n\geq 0}$ belongs to either of these categories, then
$H^*_c(X)$ if $X\in \widehat \SSet$ or $H^*(X)$ if $X\in \SSet$. We
consider cohomology groups of classifying spaces of either finite or
profinite groups. 

\begin{dfn}\label{definition:classifyingspace} Let $G$ be a finite
  group or a profinite group. Then its \emph{classifying space} $BG\in
  \SSet$ or $BG\in \widehat \SSet,$ respectively, is the simplicial set
  or simplicial profinite set, respectively, with $n$-simplices
  $(BG)_n=G\times \ldots \times G$ ($n$-copies). 
\end{dfn}

Here, the face and degeneracy maps are the usual ones, and we have
ring isomorphisms
\begin{equation}\label{equ:cohoG_is_cohoBG}
H^*(G)\cong H^*(BG)\text{ for $G$ finite and }H^*_c(G)\cong H^*_c(BG)\text{ for $G$ profinite.}
\end{equation}
Morel builds in \cite[Proposition 2]{Morel1996} a functor $\widehat
\SSet\to \widehat \SSet$, that we denote as $X\mapsto X_p$, together
with a natural transformation $X\to X_p$. The object $X_p$ is the
fibrant replacement of $X$ in a certain model category structure on
$\widehat \SSet$ for which the weak equivalences are the
$H^*_c(\cdot)$-isomorphisms. In particular, the induced map, 
\begin{equation}\label{equ:coho_fibrant_replacement}
H^*_c(X_p)\to H^*_c(X) \text{ is an isomorphism.}
\end{equation}
Sullivan's profinite completion can be obtained from Morel's construction after forgetting the profinite topology, see \cite[p. 368]{Morel1996}.

\begin{dfn}\label{definition:pcompletion_simplicial_profiniteset}
Let $X\in\widehat\SSet$ be a simplicial profinite set. Then its \emph{$p$-completion} is the simplicial profinite set $X_p$.
\end{dfn}
If $G$ is a profinite group, Equations
\eqref{equ:cohoG_is_cohoBG} and \eqref{equ:coho_fibrant_replacement}
give a ring isomorphism
\begin{equation}\label{equ:cohoBGp_is_cohoG}
H^*_c(G)\cong H^*_c(BG_p),
\end{equation}
i.e., we recover the continuous cohomology of $G$ via the
$p$-completion $BG_p$ of its classifying space. Finally, let $G$ be a
finite group and let $(\BF_p)_\infty BG$ denote the Bousfield-Kan
$p$-completion \cite{BousfieldKan1972} of its classifying space. Note
that $G$ may be considered as a profinite group and $BG$ as a
simplicial profinite set. By \cite[Corollary 3.16]{Goerss1998}, there
is a weak equivalence of simplicial sets, 
\[
(\BF_p)_\infty BG\to |BG_p|,
\]
where $|\cdot|\colon \widehat \SSet\to \SSet$ is the forgetful
functor. So Definition
\ref{definition:pcompletion_simplicial_profiniteset} extends the
notion of the Bousfield-Kan $p$-completion for classifying spaces of
finite groups to the $p$-completion for classifying spaces of
profinite groups. 


\section{Stable elements and a spectral sequence}
\label{section:stable_elements_and_a_spectral_sequence}

We start this section with the definition of a subset of morphisms that generates a 
pro-fusion system. Then we introduce the notion of stable elements in the present
context, before proving Theorems
\ref{T:introduction_Cartan_Eilenberg},
\ref{T:introduction_cohomology_profinitegroup} and
\ref{T:introduction_SpectralSequence}.  


\begin{dfn}\label{def:fin-gen}
Let $\CF$ be a pro-fusion system on a pro-$p$ group $S$ and let $\CM$
be a set of morphisms of $\CF$. We say that $\CF$ is \emph{generated
  by $\CM$} if every morphism in $\CF$ is equal to the composition of
a finite sequence of restrictions of morphisms in $\CM \cup
\Inn(S)$. In this case, we call $\CM$ a set of {\em generators} of
$\CF$. If $\CM$ is finite, we say that $\CF$ is \emph{finitely
  generated}. 
\end{dfn}

As mentioned in the introduction, $\CF^\circ$ denotes the full
subcategory of $\CF$ with objects the open subgroups of $S$.
The notions of generation and finite generation may analogously be  
defined for $\CF^{\circ}$.

\begin{lem} \label{T:F f.g. implies Fopen f.g.}
Let $\CF$ be a pro-fusion system on a pro-$p$ group $S$. If $\CF$ is
finitely generated, then $\CF^\circ$ is also finitely generated. 
\end{lem}

\begin{proof}
Let $\CM$ be a finite set of generators of $\CF$. We show that $\CM^\circ = \{
\varphi\in\CM \mid \varphi \text{ belongs to } \CF^\circ \}$ is a
finite set of generators of $\CF^\circ$. Let $\varphi\colon P\to Q$ be
an $\CF^\circ$-isomorphism. Then there exists a finite sequence of
$\CF$-isomorphisms $\varphi_i\colon P_{i-1} \to P_{i}$ ($1\leq i\leq
n$) such that $P_0 = P$, $P_n = Q$, $\varphi = \varphi_n \circ \cdots
\circ \varphi_1$, where each $\varphi_i$ is a restriction of some
$\CF$-morphism in $\CM\cup\Inn(S)$. Now the $\CF$-isomorphisms
preserve the open subgroups and also their indices in $S$ by
\cite[Lemma 2.11]{StancuSymonds2014}. Thus all the subgroups $P_i$ are
open, and hence all the morphisms $\varphi_i$ are restrictions of some
$\CF^\circ$-morphisms in $\CM\cup\Inn(S)$. This makes $\CM^\circ$ a
finite set of generators of $\CF^\circ$. 
\end{proof}

Next we define the $\CF$-stable elements for general functors with
domain $\CF$.

\begin{dfn}\label{D:GeneralStablElements}
Let $\CF$ be a pro-fusion system on a pro-$p$ group $S$ and consider a
functor $H\colon \CF^{op}\to \Ab$ from the category $\CF$ to the
category of abelian groups. We say that $x \in H(S)$ is {\em
  $\CF$-stable} (resp.\ {\em $\CF^\circ$-stable}) if $H(\varphi)(x) =
H(\iota^S_P)(x)$ for all closed (resp.\ open) subgroups $P\leq S$ and
all $\CF$-morphisms $\varphi\colon P\to S$.  We write $H(S)^{\CF}$
(resp.\ $H(S)^{\CF^\circ}$) for the abelian subgroup of $\CF$-stable
(resp.\ $\CF^\circ$-stable) elements of $H(S)$. 
\end{dfn}

Here and throughout, $\iota^S_P$
denotes the inclusion homomorphism from $P$ to $S$, and 
we write $\res_P^S$ instead of $H(\iota_P^S)$ whenever the functor
$H$ is clear from the context. Note that the abelian group $H(S)^{\CF}$ (resp.\ $H(S)^{\CF^\circ}$)
is exactly the inverse limit $\varprojlim_{P \in \CF} H(\cdot)$
(resp.\ $\varprojlim_{P \in \CF^\circ} H(\cdot)$). The functors
$H\colon \CF^{op}\to \Ab$ that we consider satisfy the following
condition:
given any closed subgroup $P$ of $S$ and any $x\in P$, then
$H(c_x)=1_{H(P)}\colon H(P)\to H(P)$ is the identity. This is the
case for the cohomology functors used in Theorems
\ref{T:StableElements} and \ref{T:SpectralSequence} below. The next
lemma is easy to prove. It shows, in particular, that if $\CF$ is
finitely generated, it is enough to consider a finite number of
conditions to determine the $\CF$-stable elements. 

\begin{lem}\label{L:stable_elements_orbit_functor_finitely_generated_F}
Let $\CF$ be a pro-fusion system on a pro-$p$ group $S$. Assume that
$\CF$ is generated by $\CM$. Let $H\colon \CF^{op}\to \Ab$ be a
functor such that $H(c_x)=1_{H(P)}$ for $c_x\in \Inn(P)$, $P\leq
S$. Then, 
\[
H(S)^\CF=\{x\in H(S)\text{ $|$ }H(\iota^S_Q\circ \varphi)(x) = H(\iota^S_P)(x)\text{ for all }\varphi\colon P\to Q, \varphi\in \CM\}.
\]
A similar statement holds for $\CF^\circ$-stable elements if $\CF^\circ$ is generated by $\CM^\circ$.
\end{lem}

For the continuous cohomology functor $H=H^*_c(\cdot)\colon \CF\to
\BF_p\operatorname{\sf{-Alg}}\subseteq \Ab$, note that $H^*_c(S)^{\CF}$ and
$H^*_c(S)^{\CF^\circ}$ are $\BF_p$-subalgebras of $H^*_c(S)$.
The classical stable elements theorem of Cartan and
Eilenberg {\cite[XII.10.1]{CartanEilenbergBook}} states that, for a
finite group $G$ with Sylow $p$-subgroup $S$, we have $H^*(G;\BF_p)
\cong H^*(S;\BF_p)^{\CF_S(G)}$. We next show that a similar statement holds for
profinite groups,  proving Theorem
\ref{T:introduction_Cartan_Eilenberg}. 

\begin{thm}[Stable Elements Theorem for Pro-Fusion Systems] \label{T:StableElements}
Let $\CF = \varprojlim_{i\in I} \CF_i$ be a pro-fusion system on a
pro-$p$ group $S$.
Suppose that $\CF^o$ is finitely generated, or that $\CF$ is
pro-saturated, or both.
Then
\[
	H^*_c(\CF;\BF_p)=H^*_c(S;\BF_p)^\CF  = H^*_c(S;\BF_p)^{\CF^\circ} \cong \varinjlim_{i\in I} H^*(S_i;\BF_p)^{\CF_i}.
\]
\end{thm}

\begin{rmk}\label{rmk:stableelements_versions}
The conclusion of Theorem \ref{T:StableElements} using the
assumption that $\CF$ is pro-saturated was suggested by Peter Symonds
\cite{Symonds2020}.

Observe also that by Lemma~\ref{T:F f.g. implies Fopen f.g.}, the
above theorem holds under the assumption that $\CF$ is finitely
generated. 
\end{rmk}

\begin{proof}
We have an isomorphism
\[
	H^*_c(S;\BF_p) \cong \varinjlim_{i\in I} H^*(S_i;\BF_p),
\]
which is induced by the maps $\varphi_i^*\colon H^*(S_i;\BF_p) \to
H^*(S;\BF_p)$. Since $I$ is directed, every element of
$\varinjlim_{i\in I} H^*(S_i;\BF_p)$ can be represented by some $x_i
\in H^*(S_i;\BF_p)$ and two elements $x_i \in H^*(S_i;\BF_p)$ and $x_j
\in H^*(S_j;\BF_p)$ represent the same element in $\varinjlim_{i\in I}
H^*(S_i;\BF_p)$ if and only if there exists some $k \in I$ with $k
\geq i$, $k\geq j$ such that $\varphi_{ik}^*(x_i) =
\varphi_{jk}^*(x_j)$. 

First we check that the direct system $\{ H^*(S_i;\BF_p),
\varphi_{ij}^* \}$ restricts to the subsystem $\{
H^*(S_i;\BF_p)^{\CF_i}, \varphi_{ij}^* \}$; that is, if $j\geq i$,
then $\varphi_{ij}^*\colon H^*(S_i;\BF_p) \to H^*(S_j;\BF_p)$ sends
$H^*(S_i;\BF_p)^{\CF_i}$ into $H^*(S_j;\BF_p)^{\CF_j}$. To see this,
let $x_i \in H^*(S_i;\BF_p)^{\CF_i}$. We want to show that
$\varphi_{ij}^*(x_i) \in H^*(S_j;\BF_p)^{\CF_j}$. Let $\psi_j\colon
P_j \to S_j$ be a morphism in $\CF_j$. Then the commutative diagrams 
\[
\xymatrix{
	S_j \ar[r]^{\varphi_{ij}} & S_i && S_j \ar[r]^{\varphi_{ij}} & S_i \\
	P_j \ar[r]^{\varphi_{ij}|_{P_j}} \ar[u]^{\psi_j} & P_i \ar[u]_{\psi_i} && P_j \ar[r]^{\varphi_{ij}|_{P_j}} \ar[u]^{\iota_{P_j}^{S_j}} & P_i \ar[u]_{\iota_{P_i}^{S_i}} \\
}
\]
(where $P_i := \varphi_{ij}(P_j)$, $\psi_i:= F_{ij}(\psi_j)$) induce commutative diagrams
\[
\xymatrix{
	H^*(S_j;\BF_p) \ar[d]_{\psi_j^*} & H^*(S_i;\BF_p) \ar[l]_{\varphi_{ij}^*}\ar[d]^{\psi_i^*} && H^*(S_j;\BF_p) \ar[d]_{\res^{S_j}_{P_j}} & H^*(S_i;\BF_p) \ar[l]_{\varphi_{ij}^*}\ar[d]^{\res^{S_i}_{P_i}}\\
	H^*(P_j;\BF_p) & H^*(P_i;\BF_p) \ar[l]_{(\varphi_{ij}|_{P_j})^*} && H^*(P_j;\BF_p) & H^*(P_i;\BF_p) \ar[l]_{(\varphi_{ij}|_{P_j})^*} \\
}
\]
Since $x_i \in H^*(S_i;\BF_p)^{\CF_i}$, the two images of $x_i$ in
$H^*(P_i;\BF_p)$ coincide. Thus the two images of
$\varphi_{ij}^*(x_i)$ in $H^*(P_j;\BF_p)$ coincide too. This shows that $\varphi_{ij}^*(x_i)\in H^*(S_j;\BF_p)^{\CF_j}$, as desired. Thus it makes sense to consider $\varinjlim_{i\in I} H^*(S_i;\BF_p)^{\CF_i}$ as a subspace of $\varinjlim_{i\in I} H^*(S_i;\BF_p)$.

Now we check that the image of $\varinjlim_{i\in I} H^*(S_i;\BF_p)^{\CF_i}$ in $H^*_c(S;\BF_p)$ is contained in $H^*_c(S;\BF_p)^{\CF}$. Let $x_i \in H^*(S_i;\BF_p)^{\CF_i}$, representing an element of $\varinjlim_{i\in I} H^*(S_i;\BF_p)^{\CF_i}$. If $P$ is a closed subgroup of $S$ and $\psi\colon P \to S$ is an $\CF$-morphism, then the commutative diagrams
\[
\xymatrix{
	S \ar[r]^{\varphi_{i}} & S_i && S \ar[r]^{\varphi_{i}} & S_i \\
	P \ar[r]^{\varphi_{i}|_{P}} \ar[u]^{\psi} & P_i \ar[u]_{\psi_i} && P \ar[r]^{\varphi_{i}|_{P}} \ar[u]^{\iota_{P}^{S}} & P_i \ar[u]_{\iota_{P_i}^{S_i}} \\
}
\]
(where $P_i := \varphi_{i}(P)$, $\psi_i:= F_{i}(\psi)$) induce commutative diagrams
\[
\xymatrix{
	H^*_c(S;\BF_p) \ar[d]_{\psi^*} & H^*(S_i;\BF_p) \ar[l]_{\varphi_i^*}\ar[d]^{\psi_i^*} && H^*_c(S;\BF_p) \ar[d]_{\res^{S}_{P}} & H^*(S_i;\BF_p) \ar[l]_{\varphi_i^*}\ar[d]^{\res^{S_i}_{P_i}}\\
	H^*_c(P;\BF_p) & H^*(P_i;\BF_p) \ar[l]_{(\varphi_i|_{P})^*} && H^*_c(P;\BF_p) & H^*(P_i;\BF_p) \ar[l]_{(\varphi_i|_{P})^*} \\
}
\]
Since $x_i \in H^*(S_i;\BF_p)^{\CF_i}$, the two images of $x_i$ in
$H^*(P_i;\BF_p)$ coincide. Thus the two images of $\varphi_i^*(x_i)$
in $H^*_c(P;\BF_p)$ also coincide. This shows that the image of
$\varinjlim_{i\in I} H^*(S_i;\BF_p)^{\CF_i}$ in $H^*_c(S;\BF_p)$ is
contained in $H^*_c(S;\BF_p)^{\CF}$. Clearly $H^*_c(S;\BF_p)^{\CF}$ is
contained in $H^*_c(S;\BF_p)^{\CF^\circ}$. Thus it remains to show
that $H^*_c(S;\BF_p)^{\CF^\circ}$ is contained in the image of
$\varinjlim_{i\in I} H^*(S_i;\BF_p)^{\CF_i}$.

Suppose first that $\CF^o$ is finitely generated. By \cite[Lemma 3.3]{StancuSymonds2014}, we may assume that 
$\CF_i=\langle F_i(\CF)\rangle$ for all $i\in I$ (see also 
\cite[Proof of Proposition 3.7]{StancuSymonds2014}). In particular, to determine the stable subring it is enough to consider $F_i(\CF)$, i.e., $H^*(S_i;\BF_p)^{\CF_i}=H^*(S_i;\BF_p)^{F_i(\CF)}$ for all $i\in I$. Now, assume that that $x\in H^*_c(S;\BF_p)^{\CF^\circ}$.
Let $P$ be an open subgroup of $S$ and let $\psi\colon P \to S$ be an
$\CF^\circ$-morphism.
Since $I$ is directed, there is some $i\in I$ and
$x_i\in H^*(S_i;\BF_p)$ such that $x = \varphi_i^*(x_i)$.
Since $\psi^*(x) = \res^S_P(x)$, the last two previous diagrams show
that $\psi_i^*(x_i)$ and $\res^{S_i}_{P_i}(x_i)$ have the same image
under $(\varphi_i|_P)^*$.
Since $H^*_c(P;\BF_p) \cong \varinjlim_{i\in I} H^*(P_i;\BF_p)$ with
the isomorphism induced by the $(\varphi_i|_P)^*$, the elements
$\psi_i^*(x_i)$ and $\res^{S_i}_{P_i}(x_i)$ in $H^*(P_i;\BF_p)$ have
the same image under $(\varphi_{ij}|_P)^*$ for some $j\geq i$.
Thus by replacing $x_i$ by $\varphi_{ij}^*(x_i)$ we may assume that
$\psi_i^*(x_i) = \res^{S_i}_{P_i}(x_i)$. We can do this for the finite
number of generators $\psi$ of $\CF^\circ$. Since $I$ is directed,
there exist $i\in I$ and $x_i\in H^*(S_i;\BF_p)$ such that
$x =\varphi_i^*(x_i)$ and $\psi_i^*(x_i) = \res^{S_i}_{P_i}(x_i)$ for
all $\psi$.
Since $H^*(S_i;\BF_p)^{\CF_i}=H^*(S_i;\BF_p)^{F_i(\CF)}$, it follows that $x_i \in H^*(S_i;\BF_p)^{\CF_i}$, and hence $x$ is in the
image of $\varinjlim_{i\in I} H^*(S_i;\BF_p)^{\CF_i}$, finishing the proof
in the case when $\CF^o$ is finitely generated.

Finally, suppose that $\CF$ is pro-saturated. 
By \cite[Proposition 4.5]{StancuSymonds2014}, we may assume that
$\CF_i = \CF/N_i$ and hence the functors $F_i\colon \CF \to \CF_i$ and
$F_{ij}\colon \CF_j \to \CF_i$ are all surjective on objects and
morphisms.
Moreover, by \cite[Proposition 4.4]{StancuSymonds2014}, the
restriction of $F_i$ to the full subcategory $\CF^o$ on the open
subgroups, $F_i\colon \CF^{o} \to \CF_i$, is also surjective on
objects and morphisms.  
Suppose that $x\in H^*_c(S;\BF_p)^{\CF^\circ}$. Since $I$ is directed,
there is some $i\in I$ and $x_i\in H^*(S_i;\BF_p)$ such that
$x = \varphi_i^*(x_i)$.
Now, each morphism $\psi_i\in \Hom_{\CF_i}(P_i,S_i)$ arises from some
morphism $\psi\in \Hom_{\CF^o}(P,S)$ with
\begin{equation}\label{equ:chooselargestpossiblesubgroup}
N_i\subseteq P.
\end{equation} Since $\psi^*(x) = \res^S_P(x)$, the last two previous
diagrams show that $\psi_i^*(x_i)$ and $\res^{S_i}_{P_i}(x_i)$ have
the same image under $(\varphi_i|_P)^*$. Then
\begin{equation}\label{equ:elementinH(P_j)iszero}
(\varphi_{ij}|_{P_j})^*(\psi_i^*(x_i)-\res^{S_i}_{P_i}(x_i))=\psi_j^*(x_j)-\res^{S_j}_{P_j}(x_j)=0\in H^*(P_j;\BF_p)
\end{equation}
for some $j\geq i$, where $x_j=\varphi_{ij}^*(x_i)$. Repeating this
argument for the finitely many morphisms of $\CF_i$, we find $k\geq i$
such that equation \eqref{equ:elementinH(P_j)iszero} holds for $j=k$
and for every morphism $\psi_i\in \CF_i$ and, in addition, for each
$\psi_i\in \CF_i$ we have chosen $P$ satisfying
\eqref{equ:chooselargestpossiblesubgroup}.
To finish, we check that $x_k \in H^*(S_k;\BF_p)^{\CF_k}$. Any
morphism $\chi_k\in \Hom_{\CF_k}(Q_k,S)$  arises from some morphism
$\chi\in \Hom_{\CF}(Q,S)$. Then $\chi_i\in \CF_i$ and, from the choice
of $P$ for $\chi_i$ satisfying
\eqref{equ:chooselargestpossiblesubgroup}, we have that $Q_k\subseteq
P_k$. From this and Equation \eqref{equ:elementinH(P_j)iszero} applied
to $j=k$ and $\chi_i\in \CF_i$, we find that
$\psi_k^*(x_k)-\res^{S_k}_{P_k}(x_k)=0$. 
\end{proof}

The next result proves Theorem \ref{T:introduction_cohomology_profinitegroup}.

\begin{cor}\label{cor:cohomology_of_profinite_group_with_open_Sylow}
Let $G$ be a profinite group with Sylow pro-$p$ subgroup $S$. Then we have
\[
		H^*_c(G;\BF_p) \cong H^*_c(S;\BF_p)^{\CF_S(G)}  = H^*_c(S;\BF_p)^{\CF_S(G)^\circ}. 
\]
\end{cor}

\begin{proof}
Write $G = \varprojlim_{i\in I} G_i$ and $S= \varprojlim_{i\in I} S_i$. Then we have
\[
	H^*_c(G;\BF_p) \cong \varinjlim_{i\in I} H^*(G_i;\BF_p) \cong \varinjlim_{i\in I} H^*(S_i;\BF_p)^{\CF_{S_i}(G_i)},
\]
where the second isomorphism is due to the stable elements theorem for finite groups. Thus, as $\CF_S(G)$ is pro-saturated, by Theorem~\ref{T:StableElements} it follows that
\[
	\varinjlim_{i\in I} H^*(S_i;\BF_p)^{\CF_{S_i}(G_i)} \cong H^*_c(S;\BF_p)^{\CF_S(G)}  = H^*_c(S;\BF_p)^{\CF_S(G)^\circ},
\]
finishing the proof.
\end{proof}

We give an easy condition on a profinite group that ensures that the
pro-fusion system that it induces on a Sylow pro-$p$ subgroup is
finitely generated, and hence stable elements are determined by a finite number of conditions.

\begin{lem} \label{T:S open implies F_S(G) f.g.}
Let $G$ be a profinite group with an open Sylow pro-$p$ subgroup $S$. Then $\CF_S(G)$ is finitely generated.
\end{lem}

\begin{proof}
Let $T$ be a (finite) set of representatives of the right cosets of
$S$ in $G$. We claim that the finite set
$\{ c_t\colon S\cap S^t \to S\cap \ls tS\mid t\in T \}$ generates $\CF_S(G)$. To see
this, suppose that $c_g \colon P \to Q$ is an $\CF_S(G)$-morphism
given by some $g\in G$.
There exist $t\in T$ and $s\in S$ such that $g= st$. So $c_g=c_sc_t$,
where $c_t\colon P \to\ls tP$ is a
restriction of $c_t\colon S\cap S^t \to S\cap\ls tS$ and $c_s\colon
\ls tP \to Q$ is a restriction of $c_s\in\Inn(S)$. This proves the
lemma. 
\end{proof}

Corollary \ref{cor:cohomology_of_profinite_group_with_open_Sylow} for
a profinite group with an open Sylow pro-$p$ subgroup can also be
proven directly using the restriction and transfer maps for cohomology
of profinite groups and their open subgroups (see \cite[\S
  6.7]{RibesZalesskii2010}), in exactly the same way as the stable
elements theorem is proven for finite groups.

We end this section with the construction of a spectral sequence
involving a 
strongly closed subgroup of a pro-saturated fusion system on a pro-$p$
group. The idea is based on the spectral sequence for a saturated fusion
system on a finite $p$-group \cite{Diaz2014}. The next result proves Theorem
\ref{T:introduction_SpectralSequence}. 

\begin{thm}\label{T:SpectralSequence}
Let $\CF$ be a pro-saturated fusion system on a pro-$p$ group $S$ and
let $T\leq S$ be a strongly $\CF$-closed subgroup. Then there is a first quadrant
cohomological spectral sequence with second page 
\[
	E_2^{n,m}= H_c^n(S/T;H_c^m(T;\BF_p))^\CF\cong H_c^n(S/T;H_c^m(T;\BF_p))^{\CF^o}
\]
and converging to $H_c^*(\CF;\BF_p)= H^*_c(S;\BF_p)^\CF= H^*_c(S;\BF_p)^{\CF^o}.$
\end{thm}

\begin{proof}
The $\CF$-stable elements of the second page in the statement are taken with respect to the functor $\CF\to \Ab$ that sends each subgroup $P\leq S$ to the second page 
\[
H_c^n(P/P\cap T;H_c^m(P\cap T;\BF_p))
\]
of the Lyndon-Hochshild-Serre spectral sequence
$\{E_k^{*,*}(P),d_k\}_{k\geq 2}$ associated to the following short
exact sequence of profinite groups (see
Equation~\eqref{equ:LHSss_profinite}):
\[
P\cap T\to P\to P/P\cap T\cong PT/P.
\]

Fix $i\in I$ and write $S_i=S/N_i$. Then $T_i=TN_i/N_i$ is a strongly $\CF_i$-closed subgroup of $S_i$ and, according to \cite[Theorem 1.1]{Diaz2014}, there is a first quadrant spectral sequence with second page
\[
	E_2^{n,m}= H^n(S_i/T_i;H^m(T_i;\BF_p))^{\CF_i}
\]
and converging to $H(S_i;\BF_p)^{\CF_i}$. Each subgroup $P_i\leq S_i$
is of the form $P_i=P/N_i$ for $N_i\unlhd P\leq S$ and the
$\CF_i$-stable elements in the last display are taken with respect to
the functor that sends $P_i$ to the second page, 
\[
	H^n(P_i/P_i\cap T_i;H^m(P_i\cap T_i;\BF_p)),
\]
of the Lyndon-Hochshild-Serre spectral sequence $\{E_k^{*,*}(P_i),d_k\}_{k\geq 2}$ associated to the following short exact sequence of finite groups,
\[
P_i\cap T_i\to P_i\to P_i/P_i\cap T_i\cong P_iT_i/P_i.
\]

Now, for each $P\leq S$, we have $P\cap T=\varprojlim_{i\in I} P_i\cap
T_i$ and, by \cite[Theorem 1.2.5(a)]{{Wilson1998}}, $P/P\cap
T=\varprojlim_{i\in I} P_i/P_i\cap T_i$. For fixed non-negative
integers $n$ and $m$, the collection $\{H^m(P_i\cap T_i;\BF_p)\}_{i\in
  I}$ is a direct system of discrete abelian groups and hence, by
\cite[Theorem 9.7.2(e)]{Wilson1998}, we have a ring isomorphism
\[
H_c^n(P/P\cap T;H_c^m(P\cap T;\BF_p))\cong\varinjlim_{i\in I} H^n(P_i/P_i\cap T_i;H^m(P_i\cap T_i;\BF_p)).
\]
In other words, using the above notation,
\begin{equation}\label{E:E_2isdirectlimit}
E_2^{*,*}(P)=\varinjlim_{i\in I} E_2^{*,*}(P_i).
\end{equation}
By induction on $k\geq 2$, we obtain ring isomorphisms
\begin{align*}
E_{k+1}^{*,*}(P)=&H^*(E_{k}^{*,*}(P),d_k)\cong H^*(\varinjlim_{i\in I} E_k^{*,*}(P_i),d_k)\cong\\
\cong&\varinjlim_{i\in I} H^*(E_k^{*,*}(P_i),d_k)\cong\varinjlim_{i\in I} E_{k+1}^{*,*}(P_i),
\end{align*}
because the direct limit over a directed set is an exact functor, 
it commutes with the cohomology functor.
Now, the functoriality of short exact sequences, the functoriality
of the Lyndon-Hochshild-Serre spectral sequence on short exact
sequences, and the arguments in the proof of Theorem \ref{T:StableElements}, prove that
\[
{E_k^{*,*}(S)}^\CF={E_k^{*,*}(S)}^{\CF^o}\cong\varinjlim_{i\in I} {E_k^{*,*}(S_i)}^{\CF_i},
\]
for all $k\geq 2$. In particular, as ${E_{k+1}^{*,*}(S_i)}^{\CF_i}=H^*(E_k^{*,*}(S_i)^{\CF_i},d_k)$ for each $i$, we have:
\begin{align*}
E_{k+1}^{*,*}(S)^\CF=&\varinjlim_{i\in I} {E_{k+1}^{*,*}(S_i)}^{\CF_i}=\varinjlim_{i\in I} {H^*(E_k^{*,*}(S_i)^{\CF_i},d_k)}\cong\\
\cong& H^*(\varinjlim_{i\in I} E_k^{*,*}(S_i)^{\CF_i},d_k)=H^*({E_k^{*,*}(S)}^\CF,d_k), 
\end{align*}
where we have used again that direct limits over direct sets are exact
functors. Hence, $\{{E_k^{*,*}(S)}^\CF,d_k\}_{k\geq 2}$ is a
sub-spectral sequence of $\{E_k^{*,*}(S),d_k\}_{k\geq 2}$. The fact
that it converges to
$H_c(S;\BF_p)^\CF$ is proven by considering filtrations as
in \cite[Proof of Theorem 4.1]{Diaz2014}. 
\end{proof}

In the above theorem, if $\CF=\CF_S(G)$ for a profinite group $G$ with
a Sylow pro-$p$ subgroup $S$ and $T\unlhd G$, then the spectral
sequence coincides with the LHS s.s. \eqref{equ:LHSss_profinite}.


\section{Mod-$p$ cohomology of $\GL_2(\BZ_p)$}
\label{section:coho_GL2}

Let $G$ be a profinite group with a Sylow pro-$p$ subgroup $S$ 
that contains a strongly $\CF_S(G)$-closed subgroup.  
Then the spectral sequence described in Theorem \ref{T:SpectralSequence}
can be used to compute the mod-$p$ cohomology ring structure of
$G$. The computations are
easier when $\CF_S(G)$ is a finitely generated pro-fusion system, in particular when $S$ is an open subgroup of $G$ by Lemma \ref{T:S open implies F_S(G) f.g.}.
A widely studied class of profinite groups with open Sylow
$p$-subgroups is the class of {\em compact $p$-adic analytic groups}
(see \cite[Corollary 8.34]{Dixonetal1991}).
In this section, we compute the mod-$p$ cohomology ring structure of
$\GL_2(\BZ_p)$ for any odd prime $p$, and the outcome is stated
in Theorem \ref{T:introduction_cohoGL2p>=3}. 

Following \cite[Section 5.1]{Dixonetal1991}, if $G=\GL_2(\BZ_p)$, then
$G=\varprojlim_{i\geq1}G/K_i$, where the subgroups  
$K_i=\{x\in G\mid x\equiv I_2\pmod{p^i}\}$ are the
{\em congruence subgroups,}\ 
for $i\geq1$. We know that $K_1$ is a uniform pro-$p$ group of rank $4$ (see \cite[Section 4.1 and Theorem 5.2]{Dixonetal1991}). Let us set the Sylow pro-$p$ subgroup of $G$,
\[
S=\{g\in \GL_2(\BZ_p) |\;g\equiv\big(\begin{smallmatrix}1&*\\0&1\end{smallmatrix}\big)\pmod p \}.
\]
In particular $K_1$ is a strongly $\CF_S(G)$-closed subgroup of $S$. 
Moreover we have a group extension 
\begin{equation}\label{eq:extesionK1Zp}
\xymatrix{1\ar[r]&K_1\ar[r]&S\ar[r]& \BZ/p\ar[r]&1},
\end{equation} 
and the spectral sequence of Theorem \ref{T:SpectralSequence} takes the form
\begin{equation}\label{equ:ssforgl2andK1}
E_2^{n,m}= H_c^n(\BZ/p;H_c^m(K_1;\BF_p))^{\CF_S(G)}\Rightarrow H_c^{n+m}(\GL_2(\BZ_p);\BF_p).
\end{equation}
In Equation~\eqref{eq:extesionK1Zp}, the group $\BZ_p$ is the subgroup
of $G$ generated by
$h=\big(\begin{smallmatrix}1&1\\0&1\end{smallmatrix}\big)$. For $p=3$, $G$ contains $3$-torsion and the extension
\eqref{eq:extesionK1Zp} splits,
while $G$ has no $p$-torsion for $p>3$, and hence the extension does not split.

\subsection{Spectral sequence for $S$}
\label{subsec:ss_S}

We start by describing the second page of the 
spectral sequence \eqref{equ:ssforgl2andK1}. 

We write 
\[
H^*(\BZ/p)=\Lambda(u)\otimes \BF_p[v]\; \text{and}\; H_c^*(K_1)\cong \Lambda(y_{11},y_{12},y_{21},y_{22}),
\]
where the generators have degrees $|u|=1$, $|v|=2$ and
$|y_{ij}|=1$. Recall that the cohomology of $K_1$ is known because it
is a uniform pro-$p$ group (see \cite[Theorem 11.6.1]{Wilson1998} and
the subsequent discussion). The generators $y_{ij}$ are the
continuous homomorphisms $y_{ij}\in \Hom_c(K_1,\BF_p)$ mapping an
element
$g=\big(\begin{smallmatrix}1+pa & pb\\ pc & 1+pd\end{smallmatrix}\big)\in K_1$ to
\[
y_{11}(g)=\overline a\text{, }y_{12}(g)=\overline b\text{, }y_{21}(g)=\overline c\text{, }y_{22}(g)=\overline d,
\]
where $\overline{\cdot}$ 
denotes the projection map $\BZ_p\to \BZ/p$.
We calculate the action of $h^{-1}$ on $H_c^1(K_1)$, and we
obtain 
\begin{equation}\label{equ:GL2_F_p_action_H_K_1}
y_{11}\mapsto y_{11}-y_{21}\ ,\;y_{12}\mapsto
y_{11}+y_{12}-y_{21}-y_{22}\ ,\;
y_{21}\mapsto y_{21}\;\hbox{and}\;y_{22}\mapsto y_{21}+y_{22}.
\end{equation}
Using the exterior algebra structure of $H_c^*(K_1)$, the above action
can be extended to all of $H_c^*(K_1)$ as follows:
\begin{equation}\label{eq:indecomposable_H_K1}
H_c^m(K_1)\cong
\begin{cases}
J^1 &\text{for $m=0,4$}\\
J^1\oplus J^3 &\text{for $m=1,3$}\\
J^3\oplus J^3 &\text{for $m=2$,}
\end{cases}
\end{equation}
where $J^n$ denotes the $\BZ/p$-module of dimension $n$ associated to
the Jordan block of size $n\times n$. We
use the periodic $\BF_p[\BZ/p]$-free resolution in
\cite[\S I.6]{Brown1982} to compute the cohomology groups
$E_2^{n,m}=H_c^n(\BZ/p; H_c^m(K_1))$. Set $y_1=y_{11}+y_{22}$,
$y_2=y_{21}$, $y_3=y_{11}y_{21}$ and
$y_4=y_{11}y_{12}y_{21}-y_{12}y_{21}y_{22}$. Then we have, 
\[
H^n_c(\BZ/p;J^1)=\BF_p\{j\}
\quad\text{with $j=$}
\begin{cases}
1&m=0,\\
y_1&m=1,\\
y_4&m=3,\\
y_1y_4&m=4,
\end{cases}
\]
and
\[
H_c^n(\BZ/p; J^3)=
\begin{cases}
\BF_3\{j\} &n=0,\;p=3,\\
0 & n>0,\;p=3,\\
\BF_p\{j\} &p>3,\; n \; \text{even},\\
\BF_p\{\overline{j}\} & p>3,\; n \;\text{odd},
\end{cases}
\quad\text{with $j=$}
\begin{cases}
y_2&m=1,\\
y_3&m=2,\\
y_2y_1+y_3&m=2,\\
y_1y_3&m=3.
\end{cases}
\]
and with
$$
\begin{array}{rclrcl}
\overline{y_2}&=&y_{12}-y_{11}\;,\quad&\overline{y_3}&=&y_{12}y_{22}-y_{12}y_{21}\;,\quad\\
\overline{y_2y_1+y_3}&=&y_{11}y_{12}\;,\quad&\overline{y_1y_3}&=&y_{11}y_{12}y_{21}-y_{11}y_{12}y_{22}.
\end{array}
$$
To denote a class in $E_2^{n,m}$, we use the representative in $H^n_c(\BZ/p;H_c^m(K_1))$ defined above multiplied with $uv^\frac{n-1}{2}$ or $v^\frac{n}{2}$ according to the parity of $n$.

The cup products in $E_2^{*,*}$ may be computed using the diagonal
approximation (see \cite[Exercise, p. 108]{Brown1982}): if $\alpha$
and $\alpha'$ are representatives in $H_c^m(K_1)$ and $H_c^{m'}(K_1)$
of classes in $E_2^{n,m}$ and $E_2^{n',m'}$ respectively, then their
product is represented by 
\[
\begin{cases}
\alpha\alpha'&\quad\text{if $nn'$ is even, and}\\
\sum_{0\leq i<j<p} h^i(\alpha)h^j(\alpha')&\quad\text{if $nn'$ is odd.}
\end{cases}
\]

For $p=3$, we obtain that $E_2^{*,*}$ is generated by the classes
$u,v,y_1,y_2,y_3,y_4$ and it has the following presentation as a
bigraded $\BF_3$-algebra,
{\small
\begin{equation}\label{equ:E2pequal3}
\BF_3[v]\otimes \Lambda[u,y_1,y_2,y_3,y_4]/\{y_2u=y_2v=y_3u=y_3v=y_2y_3=y_2y_4=y_3y_4=0\}.
\end{equation}
}For $p>3$, we obtain that $E_2^{*,*}$ is generated by the classes
$u,v,y_1,y_2,y_3,y_4,u\overline{y_2},u\overline{y_3}$, because of the
relations
\begin{align*}
u\overline{y_1y_3}=\frac{1}{2}uy_4-y_1 u\overline{y_3}\quad\text{and}\quad u\overline{y_1y_2}=-y_1u\overline{y_2}.
\end{align*}
In both cases, multiplication by $v$ is an isomorphism from
$E_2^{n,m}$ to $E_2^{n+2,m}$ for $n\geq 1$ and $m\geq 0$.
Figure \ref{F:E2cornerLHSofSforp=3p>3} describes the bottom left
corner of the spectral sequences in \eqref{equ:ssforgl2andK1}, for $p=3$ and $p>3$ separately. In this and the subsequent tables, we
leave an entry blank if it is zero.

\begin{figure}[h]
\begin{subfigure}[t]{0.45\textwidth}
\[{\small
\xymatrix@=0pt{
&&&\\
4&y_1y_4 &uy_1y_4 &vy_1y_4 \\
3&y_4, y_1y_3 &uy_4  &vy_4 \\
2&y_3,y_1y_2  && \\
1&y_1, y_2&uy_1&vy_1\\
0&1&u&v \\
&0 & 1 & 2 \\
\ar@{-}"1,2"+<-17pt,-2pt>;"7,2"+<-17pt,-5pt>;
\ar@{-}"6,1"+<-5pt,-5pt>;"6,4"+<25pt,-5pt>;
}
}
\]
\end{subfigure}
\begin{subfigure}[t]{0.45\textwidth}
\[{\small
\xymatrix@=0pt{
&&&\\
4&y_1y_4 &uy_1y_4 &vy_1y_4 \\ 
3&y_4, y_1y_3 &uy_4, u\overline{y_1y_3} &vy_1y_3,vy_4 \\
2&y_3,y_1y_2  &u\overline{y}_3, u\overline{y_1y_2}&vy_3, vy_1y_2 \\
1&y_1, y_2&uy_1, u\overline{y}_2&vy_1,vy_2 \\
0&1&u&v \\
&0 & 1 & 2 \\
\ar@{-}"1,2"+<-17pt,-2pt>;"7,2"+<-17pt,-5pt>;
\ar@{-}"6,1"+<-5pt,-5pt>;"6,4"+<25pt,-5pt>;
}
}
\]
\end{subfigure}
\caption{The corner of $E_2$ for $p=3$ (left) and $p>3$ (right).}\label{F:E2cornerLHSofSforp=3p>3}
\end{figure}

\subsection{Continuous mod-$3$ cohomology ring of $\GL_2(\BZ_3)$.}

We set $p=3$ and we compute the ring $H_c^*(\GL_2(\BZ_3);\BF_3)$.

\begin{thm}
\label{thm:ssGp=3}
The spectral sequence \eqref{equ:ssforgl2andK1} for $p=3$ satisfies that $E_2^{\CF_S(G)}$ is the free bigraded algebra with generators
{\small 
 \[
\xymatrix@=0pt{
&&&\\
E_2^{0,m}&y_{1}& &y_{4}\\
m&1 &2 &3 \\
\ar@{-}"1,2"+<-8pt,-6pt>;"3,2"+<-8pt,-6pt>;
\ar@{-}"2,1"+<-10pt,-7pt>;"2,4"+<15pt,-7pt>;
}\quad
\xymatrix@=0pt{
&&&&\\
E_2^{n,0}& & &uv&v^2\\
n&1 &2 &3 &4\\
\ar@{-}"1,2"+<-6pt,-6pt>;"3,2"+<-6pt,-6pt>;
\ar@{-}"2,1"+<-10pt,-7pt>;"2,5"+<10pt,-7pt>;
}
\]
}and, in addition, $d_i(y_1)=d_i(y_4)=d_i(uv)=d_i(v^2)=0$ for all $i\geq 2$. Moreover, $H_c^*(\GL_2(\BZ_3))\cong \BF_3[X]\otimes \Lambda(Z_1, Z_2,Z_3)$, where $Z_1, Z_2, Z_3$ and $X$ are the liftings of $y_1, uv, y_4$ and $v^2,$ respectively; of degrees $|Z_1|=1$, $|Z_2|=|Z_3|=3$ and $|X|=4$.
\end{thm}

\begin{proof}

By Lemma \ref{T:S open implies F_S(G) f.g.}, $\CF_S(G)$ is generated
by the finite set
\[
\mathcal{M}=\{c_{g_1}, c_{g_2} : K_1\to K_1 \}\cup \{c_{g_t}, c_{g_z}: S\to S\},
\]
where, the conjugation morphisms are given by the following elements,
\[
g_1=\big(\begin{smallmatrix}1&1\\0&1\end{smallmatrix}\big)\ ,\quad
g_2=\big(\begin{smallmatrix}1&0\\1&1\end{smallmatrix}\big)\ ,\quad
g_t=\big(\begin{smallmatrix}t&0\\0&1\end{smallmatrix}\big)\ ,\quad g_z=\big(\begin{smallmatrix}1&0\\0&z\end{smallmatrix}\big),
\]
with $t,z\in\{1,2\}$. A straightforward computation shows that the elements of $E_2^{0,*}$ that restrict to $H^*_c(K_1)^{\langle g_1,g_2\rangle}$ are the following,
{\small 
\begin{equation*}\label{eq:E2stableonK1}
\xymatrix@=0pt{
&&&&\\
E_2^{0,m}&y_{1}&&y_{4}&y_1y_4\\
m&1 &2 &3&4 \\
\ar@{-}"1,2"+<-8pt,-6pt>;"3,2"+<-8pt,-6pt>;
\ar@{-}"2,1"+<-10pt,-7pt>;"2,5"+<12pt,-7pt>;
}
\end{equation*}
}

It remains to find the stable elements in $E_2^{*,*}$ under the
action of $g_t$, 
\[
u\mapsto t^{-1}u,\; v\mapsto t^{-1}v,\;
y_{1}\mapsto y_{1} \text{, }\;y_{2}\mapsto ty_{2},\;
y_{3}\mapsto ty_{3}\text{, }\;y_{4}\mapsto y_{4},\;
\] and the action of $g_z$,
\begin{align*}
u\mapsto zu, v\mapsto zv,
y_{1}\mapsto y_{1} \text{, }\;y_{2}\mapsto z^{-1}y_{2},\;
y_{3}\mapsto z^{-1}y_{3}\text{, }\;y_{4}\mapsto y_{4}.
\end{align*}

A short computation shows that $E_2^{\CF_S(G)}$ is generated by the elements $y_1, uv, y_4, v^2$, which are circled in the next figure.
\begin{equation*}\label{eq:E2GofG}
\xymatrix@=0pt{
&&&&&&&&&&&\\
4&y_1y_4 &  &  & y_1y_4uv &y_1y_4v^2 & & &y_1y_4uv^3&y_1y_4v^4& \\
3& \xybox{(0,0.5)*+[Fo]{y_4}} &  &  & y_4uv &y_4v^2& &&y_4uv^3 &y_4v^4\\
2& &  &  & && & & \\
1& \xybox{(0,0.5)*+[Fo]{y_1}}&  &  & y_1uv & y_1v^2&&&y_1uv^3 & y_1v^4\\
0&1 & &  & \xybox{(0,0.2)*+[Fo]{uv}} &  \xybox{(0,0.6)*+[Fo]{v^2}}& & &uv^3&v^4\\
&0&1&2&3&4&5&6&7&8\\
\ar@{-}"1,2"+<-10pt,-6pt>;"7,2"+<-10pt,-6pt>;
\ar@{-}"6,1"+<-15pt,-7pt>;"6,10"+<10pt,-7pt>;
}
\end{equation*}

Moreover, by Equation \eqref{equ:E2pequal3},  $E_2^{\CF_S(G)}$ is the free bigraded algebra on these generators.
Clearly, $d_i(y_1)=d_i(uv)=d_i(v^2)=0$ for all $i\geq 2$. To deduce
the differentials $d_i(y_4)$, $i\geq 2$, consider the quotient group
$Q=S/K_2$, where $K_2=\{x\in\GL_2(\BZ_3)\mid x\equiv I_2\pmod{3^2}\}$
denotes the second congruence subgroup. 

This yields a commutative diagram of split extensions,
\[
\xymatrix@R=13pt{
1\ar[r]&K_1 \ar[d]\ar[r] &S\ar[d]\ar[r] &\BZ/3 \ar@{=}[d]\ar[r]&1\\
1\ar[r]&K \ar[r] &Q\ar[r] &\BZ/3\ar[r]&1.}
\]
We have $Q=\langle k_{11},k_{12},k_{21},k_{22},h\rangle$ with $K\cong (\BZ/3)^4$ generated by the elements 
\[
k_{11}=\Big(\begin{smallmatrix}1+3&0\\0&1\end{smallmatrix}\Big),
\quad k_{12}=\Big(\begin{smallmatrix}1&3\\0&1\end{smallmatrix}\Big),\quad k_{21}=\Big(\begin{smallmatrix}1&0\\3&1\end{smallmatrix}\Big),\quad\text{and}\quad k_{22}=\Big(\begin{smallmatrix}1&0\\0&1+3\end{smallmatrix}\Big).
\]
The action of $h$ permutes the first three elements of the following
basis of $K$ cyclically, and $h$ fixes the last one:
\[
\{k_{11}-k_{12}+k_{21}-k_{22},\;k_{21},\;-k_{11}-k_{12}+k_{21}+k_{22},\;k_{11}+k_{22}\}.
\]
So $Q=(\BZ/3 \wr \BZ/3)\times \BZ/3$ and the LHS spectral sequence
associated to $Q$ collapses at the second page (see \cite[IV.1, Theorem
  1.7]{AdemMilgram2004}). Moreover, it is straightforward to see that
$y_4$ is in the image of the map $H^*(K)\to H^*_c(K_1)$, and so $d_i(y_4)=0$
for $i\geq 2$. Finally, the only possible lift of a free bigraded
algebra is the free graded algebra on the same total degrees
\cite[\S 1.5, Example 1.K]{McCleary2001}. The description of
$H^*_c(\GL_2(\BZ_3);\BF_3)$ follows. 
\end{proof}

\subsection{Continuous mod-$p$ cohomology ring of $\GL_2(\BZ_p)$ for $p>3$.}

We now compute the ring structure of $H_c^*(\GL_2(\BZ_p);\BF_p)$ for $p>3$. 

\begin{thm}
\label{thm:ssGp>3}
The spectral sequence \eqref{equ:ssforgl2andK1} for $p>3$ satisfies that
$E_{\infty}^{\CF_S(G)}$ is a free bigraded algebra with generators $y_1, vy_2$ with bigraded degrees $|y_1|=(0,1)$ and $|vy_2|=(2,1)$.
In addition, $H_c^*(\GL_2(\BZ_p);\BF_p)\cong \Lambda(Z_1,Z_2),$ where $Z_1$ and $Z_2$ are the liftings of $y_1$ and $vy_2$, respectively; of total degrees $|Z_1|=1$, $|Z_2|=3$.
\end{thm}

\begin{proof}
We follow the proof of Theorem \ref{thm:ssGp=3} and the notation there. The action of $g_t$ on the remaining generators of $E_2$ for $p>3$ (see Figure \ref{F:E2cornerLHSofSforp=3p>3}) is determined by the following equations,
$$
\begin{array}{rclrcl}
\overline{y_2}&\mapsto&
\frac{t^{-1}-1}{2}y_1+t^{-1}\overline{y_2}\text{, }&
\overline{y_3}&\mapsto& t^{-1}\overline{y_3}\text{, }\\
\overline{y_1y_2}&\mapsto& t^{-1}\overline{y_1y_2} \text{, }& \overline{y_1y_3}&\mapsto&\frac{1-t^{-1}}{2}y_4+t^{-1}\overline{y_1y_3}
\end{array}$$
and the action of $g_z$ by the next set of expressions,
\[
\overline{y_2}\mapsto\frac{z-1}{2}y_1+z\overline{y_2}\text{, } \quad\overline{y_3}\mapsto z\overline{y_3}\text{, }\quad\overline{y_1y_2}\mapsto z\overline{y_1y_2}\text{, }\quad\overline{y_1y_3}\mapsto \frac{1-z}{2}y_4+z\overline{y_1y_3}.
\]
A tedious computation shows that $E_2^{\CF_S(G)}$ is generated by 
$$
y_1,\; y_4,\;  vy_2, \;vy_3,\;
\tfrac{1}{2}uv^{p-3}y_1+uv^{p-3}\overline{y_2},\;
uv^{p-3}\overline{y_1y_2},$$
$$uv^{p-3}\overline{y_3},\;
-\tfrac{1}{2}uv^{p-3}y_4+uv^{p-3}\overline{y_1y_3},\; uv^{p-2},\; v^{p-1}
$$
with bigraded degrees 
{\small
\begin{align*}
|y_1|=(0,1), |y_4|=(3,0), |vy_2|=(2,1), |vy_3|=(2,2),  |\tfrac{1}{2}uv^2y_1+uv^2\overline{y_2}|=(2p-5,1), \\
|uv^2\overline{y_1y_2}|=(2p-5,2), |uv^{p-3}\overline{y_3}|=(2p-5,2),  \\
|-\tfrac{1}{2}uv^{p-3}y_4+uv^{p-3}\overline{y_1y_3}|=(2p-5,3), |uv^{p-2}|=(2p-3,0),|v^{p-1}|=(2p-2,0).
\end{align*}}Multiplication by $v^{p-1}$ gives an isomorphism  
\begin{equation}\label{equ:8periodic}
(E_2^{n,m})^{\CF_S(G)}\cong (E_2^{n+2(p-1),m})^{\CF_S(G)}
\end{equation} for all $n\geq 1$ and $m\geq 0$. The following figure
is the $E_2$ corner for the $p=5$ case and the circled elements are
the generators listed above. The only possibly non-trivial
differentials are $d_2$ and $d_3$ and they are also depicted in the
figure. 

{\small
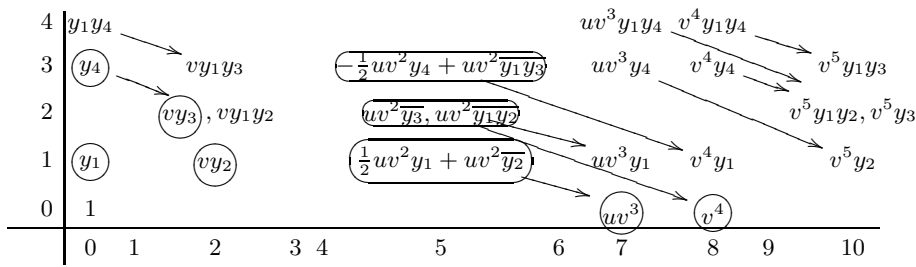
\begin{equation*}\label{eq:E2FofGp>3}
\xymatrix@=0pt{
&&&&&&&&&&&\\
4&y_1y_4\ar[rrd] &  &  &  & & & &uv^3y_1y_4\ar@<1ex>[rrrdd]&v^4y_1y_4\ar[rrd]& &&\\
3&\xybox{(0,0.5)*+[Fo]{y_4}}\ar[rrd] &  &vy_1y_3  &  &&*[F:<8pt>]{-\frac{1}{2}uv^2y_4+uv^2\overline{y_1y_3}}\ar[rrrdd] &&uv^3y_4\ar[rrrdd] &v^4y_4\ar[rrd]&&v^5y_1y_3\\
2& &  &\xybox{(0,0.2)*+[Fo]{vy_3}}, vy_1y_2 & &&*[F:<9pt>]{uv^2\overline{y_3}, uv^2\overline{y_1y_2}}\ar[rrrdd]\ar@<0.5ex>[rrd] & & &&&v^5y_1y_2, v^5y_3\\
1&\xybox{(0,0.5)*+[Fo]{y_1}}&  &\xybox{(0,0.2)*+[Fo]{vy_2}}  & & &*+[F:<7pt>]{\frac{1}{2}uv^2y_1+uv^2\overline{y_2}}\ar@<0.5ex>[rrd]& &uv^3y_1 & v^4y_1&&v^5y_2\\
0&1 & &  & &  & & &\xybox{(0,0.2)*+[Fo]{uv^3}}&\xybox{(0,0.2)*+[Fo]{v^4}}&&\\
&0&1&2&3&4&5&6&7&8&9&10\\
\ar@{-}"1,2"+<-10pt,-6pt>;"7,2"+<-10pt,-6pt>;
\ar@{-}"6,1"+<-15pt,-7pt>;"6,12"+<10pt,-7pt>;
}
\end{equation*}
}

Since $G$ is a compact $p$-adic analytic group, its cohomology ring is
finitely generated (see \cite{Lazard1965}) and, since $G$ has no
$p$-torsion, we conclude that $H_c^*(G;\BF_p)$ is a finite
$\BF_p$-algebra from Quillen's F-isomorphism Theorem
\cite{Quillen1971}. In particular, the elements in the columns
$2p-5+2(p-1)k$, $2p-3+2(p-1)k$, $2p-2+2(p-1)k$ and $2p+2(p-1)k$ must
vanish for all $k\geq k_0$ for some $k_0\geq 0$. The periodicity
in Equation~\eqref{equ:8periodic} and the Leibniz rule force $k_0=0$,
and it follows that all the
differentials between non-trivial elements of these columns are nonzero. In turn, this implies that $d_2(y_4)=\alpha vy_3
+\beta vy_1y_2$ for some $\alpha,\beta\in\BF_p$ with $\alpha\neq 0$.
Consequently the only
non-trivial elements in $E_{\infty}^{\CF_S(G)}=E_4^{\CF_S(G)}$ are
$y_1, vy_2$ and $\overline{\alpha' vy_3+ \beta' vy_1y_2}$ with
$\alpha'$ and $\beta'$ in $\BF_p$ such that the determinant
$\Big|\begin{matrix} \alpha&
\alpha'\\\beta & \beta'\end{matrix}\Big|\neq 0$. Also, the product
$vy_1y_2$ is equal to $\overline{\alpha' vy_3+ \beta' vy_1y_2}$ because
$\Big|\begin{matrix} \alpha& 0\\\beta & 1\end{matrix}\Big|=\alpha\neq
0$. Thus, $E_{\infty}^{\CF_S(G)}$ is the free bigraded $\BF_p$-algebra
with generators $y_1$ and $vy_2$ of bigraded degrees $|y_1|=(0,1)$ and
$|vy_2|=(2,1)$. Again the only lift is the free graded $\BF_p$-algebra with
generators $Z_1$ and $Z_2$ of degrees $|Z_1|=1$ and $|Z_2|=3$. 
\end{proof}

\bibliographystyle{amsplain}

\begin{thebibliography}{10}

\bibitem{AdemMilgram2004} A.~Adem, R.J.~Milgram, \emph{Cohomology of finite groups}, 
Springer-Verlag Berlin Heidelberg, 2004. 

\bibitem{Aguade1980} J.~Aguad\'e, \emph{The cohomology of the $\GL_2$ of a finite field}, 
Archiv der Mathematik 34, 1980, 509--516.

\bibitem{Alperin1986} J.~Alperin \emph{Local representation theory}, 
Cambridge University Press, Cambridge, 1986.

\bibitem{Benson2004} D.~Benson, \emph{Commutative algebra in the cohomology of groups}, 
Trends in commutative algebra, 1–50, Math. Sci. Res. Inst. Publ., 51, Cambridge Univ. Press, Cambridge, 2004.

\bibitem{BLOsurvey} C.~Broto,R.~Levi, B.~Oliver, \emph{The theory of $p$-local groups: a survey}, \emph{Homotopy theory: relations with algebraic geometry, group cohomology, and algebraic $K$-theory}, 
Contemp. Math., vol. 346, Amer. Math. Soc., Providence, RI, 2004, pp. 51--84.

\bibitem{BousfieldKan1972} A.K.~Bousfield and D.M.~Kan, \emph{Homotopy Limits, Completions and Localizations}, 
Lecture Notes in Mathematics 304, Springer-Verlag 1972. 

\bibitem{Brown1982} K.S.~Brown, \emph{Cohomology of groups}, 
Graduate Texts in Mathematics, 87. Springer-Verlag, New York-Berlin, 1982. 

\bibitem{CartanEilenbergBook} H.~Cartan and S.~Eilenberg, \emph{Homological algebra}, 
Princeton University Press, Princeton, N. J., 1956.

\bibitem{Diaz2014} A.~D\'iaz Ramos, \emph{A spectral sequence for fusion systems}, Algebr. Geom. Topol. 14, 2014, no. 1, 349--378. 

\bibitem{DiazGaraialdeGonzalezSanches2017} A.~D\'iaz Ramos, O.~Garaialde Oca\~na, J. Gonz\'alez-S\'anchez, \emph{Cohomology of uniserial p-adic space groups}, 
  Trans. Amer. Math. Soc. 369, 2017, 6725-6750. 

\bibitem{Dixonetal1991}  J.~Dixon, M.~du Sautoy, A.~Mann and D.~Segal
  {\em Analytic pro-$p$-groups}, London Mathematical Society Lecture
  Note Series, 157, Cambridge University Press, Cambridge, 1991.

\bibitem{GRS1999} A.L.~Gilotti, L.~Ribes and L.~Serena,\emph{Fusion in profinite groups}, 
Ann. Mat. Pura Appl. (4), 177, 1999 349--362.


\bibitem{Goerss1998} P.~Goerss, \emph{Comparing completions of a  space at a prime}, 
Homotopy theory via algebraic geometry and group representations (Evanston, IL, 1997), 
Contemp. Math., 220, Amer. Math. Soc., Providence, RI, 1998, 65--102.

\bibitem{Gonzalez2016} A.~Gonz\'alez, \emph{Finite approximations of $p$-local compact groups}, 
Geom. Topol., Volume 20, Number 5, 2016, 2923-2995.

\bibitem{Henn1998} H.W.~Henn, \emph{Centralizers of elementary abelian $p$-subgroups and mod-$p$ cohomology of profinite groups}, 
Duke Math. J., 91, 1998, no. 3, 561--585.

\bibitem{Lazard1965} M.~Lazard, {\em Groupes analytiques $p$-adiques},
Pub. math. IH\`ES, tome 26, 1965, 5--219.

\bibitem{Leary1991} I.J.~Leary, \emph{The mod-p cohomology rings of some p-groups}, 
Math. Proc. Cambridge Philos. Soc. 112, 1992, no. 1, 63--75. 

\bibitem{McCleary2001} J.~McCleary, {\it A User's Guide to Spectral Sequences}, 
Cambridge Studies in Advanced Mathematics 58 (2nd ed.), 
Cambridge University Press, ISBN 978-0-521-56759-6, MR 1793722 (2001).

\bibitem{Morel1996} F.~Morel, \emph{Ensembles profinis simpliciaux et interpr\'etation g\'eom\'etrique du foncteur T}, 
Bulletin de la SMF 124, 1996, fac. 2, pp. 347-373. 

\bibitem{Quillen1971} D.~Quillen, \emph{The spectrum of an equivariant cohomology ring I, II}, 
Annals of Math. No. 94, 1971, pp. 549--572.

\bibitem{Quillen1972} D.~Quillen, \emph{On the Cohomology and K-Theory of the General Linear Groups Over a Finite Field}, 
Annals of Mathematics Second Series, Vol. 96, No. 3, 1972, pp. 552--586.

\bibitem{RibesZalesskii2010} L.~Ribes and P.~Zalesskii,
  \emph{Profinite groups}, Second edition. Results in Mathematics and
  Related Areas. 3rd Series. A Series of Modern Surveys in
  Mathematics, 40. Springer-Verlag, Berlin, 2010. 

\bibitem{Siegel1996} S.F.~Siegel, \emph{The spectral sequence of a split extension and the cohomology of an extraspecial group of order $p^3$ and exponent $p$}, 
J. Pure Appl. Algebra 106, 1996, no. 2, 185--198. 

\bibitem{StancuSymonds2014} R.~Stancu and P.~Symonds, \emph{Fusion systems for profinite groups}, 
J. Lond. Math. Soc. (2) \textbf{89} (2014), no.~2, 461--481.

\bibitem{Symonds2020}  P.~Symonds, \emph{Cohomology of profinite groups of bounded rank}, 
preprint 2020.

\bibitem{Wilson1998} J.S.~Wilson, \emph{Profinite groups}, 
London Mathematical Society Monographs. New Series, 19. The Clarendon
Press, Oxford University Press, New York, 1998. 


\end{thebibliography}

\end{document}